\documentclass[11pt]{article}
\usepackage{times}
\usepackage{geometry}
\usepackage{amsfonts,amsmath,amsthm,amssymb,latexsym}
\newtheorem{theorem}{Theorem}

\newtheorem{corollary}{Corollary}

\newtheorem{remark}{Remark}

\author{
Jos\'e A. S. Pelegr\'in${}^{a*}$, Alfonso Romero${}^{a*}$ and Rafael M. Rubio${}^{b*}$\\[6mm]
${}^a$ Departamento de Geometr\'{\i}a y Topolog\'{\i}a, \\
Universidad de Granada, 18071 Granada, Spain \\ E-mail\textup{:
\texttt{jpelegrin@ugr.es, aromero@ugr.es}}\\[3mm]
${}^b$ Departamento de Matem\'aticas, Campus de Rabanales, \\
Universidad de C\'ordoba, 14071 C\'ordoba, Spain\\
E-mail\textup{:\texttt{\;rmrubio@uco.es}}}

\date{}

\begin{document}

\title{On maximal hypersurfaces in Lorentz manifolds admitting a
parallel lightlike vector field}

\maketitle

\thispagestyle{empty}

\begin{abstract}
We study constant mean curvature spacelike hypersurfaces and in particular maximal hypersurfaces immersed in pp-wave spacetimes satisfying the timelike convergence condition. We prove the non-existence of compact spacelike hypersurfaces whose constant mean curvature is non-zero and also that every compact maximal hypersurface is totally geodesic. Moreover, we give an extension of the classical Calabi-Bernstein theorem to this class of pp-wave spacetimes.
\end{abstract}


\section{Introduction}
The pp-wave type spacetimes are an important family of spacetimes. Considering  such a spacetime as an exact solution of Einstein's field equation, it can model radiation (electromagnetic or gravitational) moving at the speed of light. The recent interest on pp-wave type spacetime can be explained, on the one hand, by its classical geometrical properties, on the other by its applications to string theory as well as the possibilities of direct detection of gravitational waves. Historically, the study of gravitational waves goes back to Einstein (\cite{E-R})
but the standard exact model was already introduced by Brinkmann in order to
determine Einstein spaces which can be improperly mapped conformally on some
Einstein one (see \cite{Br}). The original definition by Brinkmann says that a pp-wave spacetime is any spacetime whose metric tensor can be described, with respect to suitable coordinates, in the form
$ds^2 = H(u,x,y) \, du^2 + 2 \, du \, dv + dx^2 + dy^2,$
where H is any smooth function. Nowadays, pp-wave means any spacetime which admits a  parallel global  lightlike vector field \cite[45, p. 383]{book}.  Moreover, when the spacetime is Ricci-flat (vacuum solutions) it is called gravitational plane wave. The interest in better understanding the geometry of pp-waves and their potential applications to string theory justifies the study of these spacetimes from a wider perspective.

On the other hand, spacelike hypersurfaces in $(n\geq 3)$-dimensional Lorentzian
spacetimes are geometrical objects of great physical and
mathematical interest. Roughly speaking, each of them represents the
physical space in a given instant of a time function. In
Electromagnetism, a spacelike hypersurface is an initial data set
that univocally determines the future of the electromagnetic field
which satisfies the Maxwell equations \cite[Thm. 3.11.1]{Sa-Wu} and analogously,
for the simple matter equations \cite[Thm. 3.11.2]{Sa-Wu}. In Causality
Theory, the mere existence of a particular spacelike hypersurface
implies that the spacetime obeys a certain causal property. Let us
remark that the completeness of a spacelike hypersurface is required
whenever we study its global properties, and also, from a physical
viewpoint, completeness implies that the whole physical space is
taken into consideration. To know more details about the relevance of constant mean curvature spacelike hypersurfaces in General Relativity see \cite{M-T}.

From a mathematical point of view, the interest of spacelike
hypersurfaces is motivated by their nice Bernstein-type properties.
In fact, maximal spacelike hypersurfaces  in
an $(n\geq 3)$-dimensional spacetime are critical points of the area
functional. In the history of research on maximal hypersurfaces, a striking
fact was the discovery of new nonlinear elliptic problems. In
fact, the function defining a maximal graph in the
$(n+1)$-dimensional Lorentz-Minkowski spacetime $\mathbb{L}^{n+1}$
satisfies an elliptic second order PDE similar to the equation of
minimal graphs in Euclidean space $\mathbb{R}^{n+1}$, but with a
new and surprising behavior for its entire solutions: The only
entire solutions to the maximal hypersurface equation in
$\mathbb{L}^{n+1}$ are the affine functions defining spacelike
hyperplanes. This result was previously proved by Calabi
\cite{Calabi} for $n\leq 4$ and later extended for any $n$ in the
seminal paper by S.Y. Cheng and S.T. Yau \cite{Cheng-Yau} and it is called
Calabi-Bernstein theorem. Recall that the Bernstein theorem for
minimal graphs in the Euclidean space $\mathbb{R}^{n+1}$, holds only for
$n\leq 7$, \cite{S-S-Y}. An important goal in \cite{Cheng-Yau} was
the introduction of a new tool, the so-called Omori-Yau
generalized maximum principle \cite{OM}, \cite{YA}. By means of
the use of this technique, many uniqueness results were obtained.
For instance, S. Nishikawa \cite{N} proved that a complete maximal
hypersurface in a locally symmetric Lorentzian manifold whose
Ricci tensor satisfies a natural assumption on timelike tangent
vectors must be totally geodesic.

The special case of maximal surfaces in
$3$-dimensional spacetimes is of mathematical interest and,
sometimes, a source of inspiration for higher dimensions. So, for the $3$-dimensional Lorentz-Minkowski spacetime several new proofs of the so called classical Calabi-Bernstein theorem have been given by different authors (see \cite{Ro}, \cite{A-P2}, \cite{Ro-Ru2}, \cite{ARR}). Moreover, several extensions of the classical results in different direction have been proved recently (see \cite{AA}, \cite{E-R1}, \cite{Ca-Ro-Ru}).

In this work, we deal with complete maximal hypersurfaces in Lorentz manifolds admitting a
parallel lightlike vector field satisfying the timelike convergence condition (and in particular for Ricci-flat case). We show (see Theorem \ref{th1}) that in these Lorentzian manifolds (in fact, spacetimes) there are no compact constant mean curvature spacelike hypersurfaces except for the maximal case. Moreover, we prove that every compact maximal hypersurface must be totally geodesic. On the other hand, in section \ref{Calabi} we give a new extension of Calabi-Bernstein theorem for these Lorentzian manifolds. 

\section{Preliminaries}
The notion of symmetry is basic in Physics. In General Relativity, symmetry is usually based on the assumption of the existence of a one-parameter group of transformations generated by a Killing or, more generally, a conformal vector field. In fact, an usual simplification for the search of exact solutions to the Einstein equation is to assume a priori the existence  of such an infinitesimal symmetry (see \cite{Da}, \cite{Ea} for instance). A complete general approach to symmetries can be found in \cite{Za} (see also \cite{Du} and references therein).

A vector field $V$ on a Lorentzian manifold $(M,g)$ is called {conformal} if ${\cal L}_V g = 2 \lambda g$ for a smooth function $\lambda$, where ${\cal L}$ denotes the Lie derivative on $M$. It is easy to see that if $V$ is locally a gradient field (closed), then this condition is equivalent to

\begin{equation}
\label{confvec}
 \nabla_X V = \lambda X ,\ \forall X \in \mathfrak{X}(M)
\end{equation}

\noindent Vector fields satisfying (\ref{confvec}) are called closed conformal vector fields.
If $V \in \mathfrak{X}(M)$ is lightlike on $(M,g)$, then from (\ref{confvec}) we get

$$X g(V,V) = 2 g(\nabla_X V, V) = 2 \lambda g(X, V)=0 $$

\noindent for every $X \in \mathfrak{X}(M)$. Thus, $V$ is parallel, since $\lambda V=0$ and then $\lambda$ must be identically $0$.

\vspace{2mm}

Note that if a Lorentzian manifold $(M,g)$ admits a parallel lightlike vector field then it is time-orientable and if we choose a time-orientation and $(M,g)$ is connected then it is a spacetime.

On the other hand, recall that a Lorentzian manifold obeys the timelike convergence
condition (TCC) if its Ricci tensor $\overline{\mathrm{Ric}}$
satisfies
\[
\overline{\mathrm{Ric}}(Z,Z)\geq 0,
\]
for all timelike vector $Z$. It is normally argued that TCC is the
mathematical way to express that gravity, on average, attracts.

\section{Principal result}

Let $(\overline{M}, \langle, \rangle)$ be a Lorentzian manifold of dimension $n+1$ and let $x: M \longrightarrow \overline{M}$ be an isometrically immersed spacelike hypersurface. Let us suppose the existence of a parallel lightlike vector field $\xi$ on $\overline{M}$, i.e., $\langle \xi, \xi \rangle =0$, $\xi \neq 0$ such that $ \overline{\nabla} \xi =0$, where $\overline{\nabla}$ denotes the Levi-Civita covariant differential in $(\overline{M}, \langle, \rangle)$. We can decompose $\xi$ into its tangent and normal part, having $\xi = \xi^T + \xi^\perp = \xi^T - \langle N, \xi \rangle N$, being $N$ the timelike unitary normal vector field to $M$. It is clear that $\xi^\perp$ never vanishes. Observe that the existence  of a vector field which never vanishes can give some topological 	obstructions when the hypersurface $M$ is compact.

\begin{remark} {\rm Consider the Lorentzian manifold $(\overline{M}, \langle, \rangle)$ where $\overline{M}$ is given by the product $\Sigma\times \mathbb{S}^1\times \mathbb{R}$, endowed with the Lorentzian metric $\langle\, ,\rangle=\langle\,,\rangle_{_\Sigma}+2d\alpha dv+H(x,\alpha)d\alpha^2$, where $x\in\Sigma$, being $(\Sigma,\langle\,,\rangle_{_\Sigma})$  a compact Riemannian manifold, $\mathbb{S}^1$ the unitary sphere, $\alpha$ an angular coordinate on $\mathbb{S}^1$ and $v\in \mathbb{R}$. Now, it is clear that the vector field $\frac{\partial}{\partial v}$ on $\overline{M}$ is a parallel lightlike vector field and $\overline{M}$ admits a compact spacelike hypersurface, as long as $H$ is positive.}
\end{remark}

Let us now consider the function $\eta = \langle N, \xi \rangle$ and choose $N$ such that $\langle N, \xi \rangle >0$. It is not difficult to see that  
\begin{equation}
\label{nablau}
\nabla \eta = - A \xi^T,
\end{equation}
where $\nabla$ denotes the gradient operator on $(M,\langle\,,\rangle)$ and $A$ the shape operator of the isometric immersion associated with $N$.

Take a local orthonormal reference frame $\{ E_1, \dots, E_n\}$ for $(M,\langle\,,\rangle)$. Since $\xi$ is parallel, taking the tangent component of $\overline{\nabla}_{E_i} \xi$, we obtain

\begin{equation}
\label{nabeixi}
\nabla_{E_i} \xi^T = - \langle N, \xi \rangle A E_i. \\
\end{equation}

We compute now ${\rm{div}} (A \xi^T)$,

$$ {\rm{div}} (A \xi^T) = \sum_{i=1}^n \langle \nabla_{E_i} (A \xi^T), E_i \rangle  = \sum_{i=1}^n \langle (\nabla_{E_i} A) \xi^T, E_i \rangle + \sum_{i=1}^n \langle A(\nabla_{E_i} \xi^T), E_i \rangle.$$

Using (\ref{nabeixi}), we have

\begin{gather*}
 {\rm{div}} (A \xi^T) = \sum_{i=1}^n \langle (\nabla_{E_i} A) \xi^T, E_i \rangle + \sum_{i=1}^n \langle A(- \langle N, \xi \rangle A E_i), E_i \rangle = \\= \sum_{i=1}^n \langle (\nabla_{E_i} A) \xi^T, E_i \rangle - \langle N, \xi \rangle \sum_{i=1}^n \langle A^2 E_i, E_i \rangle = \sum_{i=1}^n \langle (\nabla_{E_i} A) \xi^T, E_i \rangle - \langle N, \xi \rangle \rm{trace}(A^2).
\end{gather*}

The Codazzi equation allows us to obtain,

\begin{gather*}
 {\rm{div}} (A \xi^T) = \sum_{i=1}^n \langle (\nabla_{\xi^T} A) E_i, E_i \rangle + \sum_{i=1}^n \langle \overline{{R}}(\xi^T, E_i) N, E_i \rangle - \langle N, \xi \rangle \rm{trace}(A^2) = \\ = \sum_{i=1}^n \langle (\nabla_{\xi^T} A) E_i, E_i \rangle -\overline{\rm{Ric}}(\xi^T, N) - \langle N, \xi \rangle \rm{trace}(A^2), 
\end{gather*}
Where $\overline{R}$ and $\overline{\rm{Ric}}$  denote the curvature and Ricci tensors of $\overline{M}$ respectively.
Finally, taking into account that tensor derivations commute with contractions we get

\begin{equation}
\label{divaxi}
{\rm{div}}(A \xi^T) = -n \langle \nabla {\cal H}, \xi \rangle - \overline{\rm{Ric}}(\xi^T, N) - \langle N, \xi \rangle \rm{trace}(A^2),
\end{equation}

\noindent where ${\cal H}=-\frac{1}{n}{\rm trace}\,(A)$ is the mean curvature function of $M$.
From (\ref{nablau}) and (\ref{divaxi}), the Laplacian of the distinguished function $\eta$ is given by

\begin{equation}
\label{laplacu}
\Delta \eta = n \langle \nabla {\cal H}, \xi \rangle + \overline{\rm{Ric}}(\xi^T, N) + \eta \ \rm{trace}(A^2).
\end{equation}

Moreover, since $\xi$ is parallel, the Ricci tensor can be expressed as ${\rm{\overline{Ric}}}(\xi^T, N) = {\rm{\overline{Ric}}}(\xi, N) + \langle N, \xi \rangle \overline{\rm{Ric}}(N, N) = \eta \ \overline{\rm{Ric}}(N, N)$ and (\ref{laplacu}) leads to

\begin{equation}
\label{laplacu2}
\Delta \eta = n \langle \nabla {\cal H}, \xi \rangle + \eta \{ \overline{\rm{Ric}}(N, N) + \rm{trace}(A^2) \}.
\end{equation}

\begin{theorem}\label{th1} Let $(\overline{M}, \langle, \rangle)$ be an $(n+1)$-dimensional ($n\geq 2$) Lorentzian manifold which admits a parallel global lightlike vector field and suppose that 
$\overline{\mathrm{Ric}}(Z,Z)\geq 0,$
for every timelike vector $Z$. Consider $x:M\longrightarrow\overline{M}$ a compact isometrically immersed spacelike hypersurface. If $M$ has constant mean curvature, then it must be totally geodesic. As a direct consequence, there is no compact constant mean curvature spacelike immersed hypersuface whose mean curvature is  different from zero.
\end{theorem}

\begin{proof}
From (\ref{laplacu2}) we have that the function $\eta$ is positive and subharmonic on a compact manifold. Thus, $\eta$ is constant. Again, from (\ref{laplacu2}) we obtain that ${\rm trace}\,(A^2)=0$.
\end{proof}

\section{A new extension of the Calabi-Bernstein theorem}\label{Calabi}

In this section we consider a (2+1)-dimensional Lorentzian manifold $(\overline{M}, \langle, \rangle)$, which admits a  parallel global lightlike vector field. Several examples of such Lorentzian manifolds can be found in \cite{Rio}. Moreover, the differentiable manifold $\mathbb{R}^ 3=\{(u,v,x): u,v,x\in \mathbb{R}\}$ can be endowed with the family of Lorentzian metrics given by

\begin{equation}\label{metr}
\langle\,,\rangle=H(u,x) \, du^2 + 2 \, du \, dv + dx^2.
\end{equation}

\noindent When the function $H$ is constant we have a realization of the Lorentz-Minkowski spacetime $\mathbb{L}^3$. Nonetheless, taking into account that the Ricci tensor (see \cite[Section 2.2]{F-S}) of this family of metrics is given by 
$$\overline{{\rm Ric}}=-\frac{\partial^2H}{\partial x^2}du\otimes du,$$ \noindent  it is clear that this family admits spacetimes which are not isometric to $\mathbb{R}^3$ and admit a global parallel lightlike vector field. So, it is enough to consider a function $H$ such that $\frac{\partial^2H}{\partial x^2}(p)\not=0$, for some point $p$.

Futhermore, we can give examples of closed maximal hypersurfaces in spacetimes of the previous family, which are not isometric to $\mathbb{L}^3$. Indeed, let $F:\mathbb{R}^ 3\longrightarrow\mathbb{R}$ be a smooth function and let $c\in \mathbb{R}$ be a regular value of $F$. It is well-known that the level set $\Sigma=F^{-1}(c)$ is a closed embedded hypersurface in $\mathbb{R}^3$. On the other hand, this hypersurface is spacelike with the induced metric, if and only if the vector field $\overline{\nabla}F$ is timelike on $\Sigma$. Taking into account the Christoffel symbols of the metric (see \cite[Section 2]{CFS}), it is easy to obtain

$$\overline{\nabla}F=\frac{\partial F}{\partial v}\frac{\partial}{\partial u}+\left(\frac{\partial F}{\partial u}-H\frac{\partial F}{\partial v}\right)\frac{\partial}{\partial v} + \frac{\partial F}{\partial x}\frac{\partial}{\partial x}.$$

\noindent Therefore, $\Sigma$ is spacelike if and only if the inequality
\begin{equation}\label{spatial}
2\frac{\partial F}{\partial u}\frac{\partial F}{\partial v}-H\left(\frac{\partial F}{\partial v}\right)^2+\left(\frac{\partial F}{\partial x}\right)^2<0,
\end{equation}

\noindent holds on $\Sigma$. 

Now, if we denote by $A$ the shape operator on $\Sigma$ associated to the unitary normal vector field $\frac{1}{\mid\overline{\nabla}F\mid}\overline{\nabla}F$, we have that 
\begin{equation}\label{hesiano}
\langle A w_1,w_2\rangle=-\frac{\overline{{\rm Hess}}(F)(w_1,w_2)}{\mid\overline{\nabla}F\mid},
\end{equation}

\noindent where $\overline{{\rm Hess}}(F)$ denotes the Hessian of the function $F$ in $(\mathbb{R}^3,\langle\,,\rangle)$ and $w_1,w_2\in T_p\Sigma$.

If we take $W_1,W_2\in\mathfrak{X}(\mathbb{R}^3)$ and we  put 

$$W_1=w_1^u\frac{\partial}{\partial u}+w_1^v\frac{\partial}{\partial v}+w_1^x\frac{\partial}{\partial x},\ \ W_2=w_2^u\frac{\partial}{\partial u}+w_2^v\frac{\partial}{\partial v}+w_2^x\frac{\partial}{\partial x},$$ \noindent then

\begin{equation}\label{hess}
\overline{{\rm Hess}}(F)(W_1,W_2)=\langle\overline{\nabla}_{W_1}(\overline{\nabla}F),W_2\rangle
\end{equation}

$$=W_1\left(\frac{\partial F}{\partial v}\right)(Hw_2^u+w_2^v)+ \frac{\partial F}{\partial v}\left(\frac{1}{2}w_1^uw_2^u \frac{\partial H}{\partial u}-\frac{1}{2}w_1^uw_2^x\frac{\partial H}{\partial x}     +\frac{1}{2}w_1^xw_2^u \frac{\partial H}{\partial x}\right)$$

$$+W_1\left(\frac{\partial F}{\partial u}-H\frac{\partial F}{\partial v}\right)w_2^u+W_1\left(\frac{\partial F}{\partial x}\right)w_2^x+\frac{\partial F}{\partial x}\left(\frac{1}{2}w_1^uw_2^u\frac{\partial H}{\partial x}\right).$$

Now, consider the spacetime $(\mathbb{R}^3, \langle,\,\rangle)$, whose metric is given in (\ref{metr}), taking $H$ a positive function such that $\frac{\partial H}{\partial u}=0$ and $\frac{\partial^2H}{\partial x^2}(p)\not=0$, for some point $p\in \mathbb{R}^ 3$. Let $F(u,v,x)=\lambda v$, $\lambda\in\mathbb{R}^+$ be. Taking into account (\ref{spatial}) and (\ref{hess}), it is not difficult to see that the hypersurface $F^{-1}(c)$ is spacelike and in particular maximal, since from (\ref{hesiano}), it is clear that ${\rm trace}\,(A)=0$.

\vspace{5mm}

Consider an immersed maximal surface $S$ in a (2+1)-dimensional Lorentzian manifold $(\overline{M}, \langle, \rangle)$, which admits a  parallel global lightlike vector field. Here, taking into account that ${\cal H}=0$, we have from (\ref{laplacu2}),

\begin{equation}
\label{laplacu3}
\Delta \eta = \eta \{ \overline{\rm{Ric}}(N, N) + \rm{trace}(A^2) \}.
\end{equation}

\noindent Moreover, in this case, using Cayley-Hamilton theorem we get

\begin{equation}
\label{modcuanabu}
 |\nabla \eta|^2 = \langle A^2 \xi^T, \xi^T \rangle = \frac{1}{2} \rm{trace}(A^2) |\xi^T|^2 = \frac{1}{2} \rm{trace}(A^2) \eta^2. 
\end{equation}

Now, we can state,

\begin{theorem}\label{geod} Let $(\overline{M}, \langle, \rangle)$ be a (2+1)-dimensional Lorentzian manifold $(\overline{M}, \langle, \rangle)$, which admits a parallel global lightlike vector field and suppose that 
$\overline{\mathrm{Ric}}(Z,Z)\geq 0,$
for all timelike vector $Z$. Let $x:S\longrightarrow\overline{M}$ be a complete isometrically immersed maximal surface. Then $S$ admits a positive function $\eta$ which is constant if and only if $S$ is totally geodesic.
\end{theorem}

On the other hand, taking into account the Gauss curvature of the surface, we can enunciate,

\begin{theorem}
\label{teototgeo}
Let $(\overline{M}, \langle , \rangle)$ be a $(2+1)$-dimensional manifold which admits a parallel global lightlike vector field. Let $x : S \longrightarrow \overline{M}$ be an isometrically immersed maximal surface. Then, 

\vspace{2mm}
\noindent (i) the Gaussian curvature of $S$ satisfies $ K \geq \overline{\rm{Ric}}(N, N)$, being $N$ the timelike unitary normal vector field to $S$. Moreover, equality holds if and only if $S$ is totally geodesic.

\noindent (ii) if the spacetime satisfies the TCC and  the surface is complete, then $S$ is parabolic.
\end{theorem}

\begin{proof} 
Since $S$ is maximal, from the Gauss equation we obtain,

\begin{equation}
\label{gauss1}
{\rm{Ric}}(\xi^T, \xi^T) = \overline{\rm{Ric}}(\xi^T, \xi^T) + \langle \overline{\rm{R}}(\xi^T, N) N, \xi^T \rangle + \langle A^2 \xi^T, \xi^T \rangle .\\
\end{equation}

Moreover, being ${\rm dim}S=2$ we have 
\begin{equation}\label{ricxit}
{\rm{Ric}}(\xi^T, \xi^T) = K |\xi^T|^2 = K \eta^2 
\end{equation}
and as  $\overline{\nabla}\xi = 0$, then

\begin{equation}\label{rricx}
\overline{\rm{Ric}}(\xi^T, \xi^T) =\eta^2 \overline{\rm{Ric}}(N, N),
\end{equation}

\noindent and

\begin{equation}
\label{rrxin}
\langle \overline{\rm{R}}(\xi^T, N) N, \xi^T \rangle = \langle \overline{\rm{R}}(\xi, N) N, \xi \rangle + 2 \eta \langle \overline{\rm{R}}(\xi, N) N, N \rangle + \eta^2 \langle \overline{\rm{R}}(N, N) N, N \rangle = 0. \\
\end{equation}

Therefore, substituting (\ref{ricxit}), (\ref{rricx}), (\ref{rrxin}) and (\ref{modcuanabu}) in (\ref{gauss1}) we have

\begin{equation}
\label{gaussguay}
K = \overline{\rm{Ric}}(N,N) + \frac{1}{2} \rm{trace}(A^2) \\
\end{equation}
\noindent and as a direct consequence (i) holds.

Finally, the TCC guarantees $K\geq 0$ and by a classical result of Ahlfors and Blanc-Fiala-Huber (see for instance \cite{Kazdan}) a complete Riemannian surface with non-negative Gaussian curvature is parabolic.
\end{proof}

We can give the following rigidity's result,

\begin{theorem} Let $(\overline{M}, \langle, \rangle)$ be a (2+1)-dimensional Lorentzian manifold $(\overline{M}, \langle, \rangle)$, which admits a parallel global lightlike vector field and suppose that 
$\overline{\mathrm{Ric}}(Z,Z)\geq 0,$
for all timelike vector $Z$. Let $x:S\longrightarrow\overline{M}$ be a complete isometrically immersed maximal surface. Then  $S$ is totally geodesic.
\end{theorem}

\begin{proof} Making use of (\ref{laplacu3}) and (\ref{modcuanabu}),
$$ \Delta \left( \frac{1}{\eta} \right) = - \frac{1}{\eta^2} \Delta \eta + \frac{2}{\eta^3} |\nabla \eta|^2 = - \frac{1}{\eta} \{ \overline{\rm{Ric}}(N, N) + \rm{trace}(A^2) \} + \frac{1}{\eta} \rm{trace}(A^2) $$

\noindent Thus,

\begin{equation}
\label{lapunopartu}
\Delta \left( \frac{1}{\eta} \right) = - \frac{1}{\eta} \overline{\rm{Ric}}(N,N).
\end{equation}

\noindent Since $S$ is parabolic due to Theorem \ref{teototgeo}, it is enough to obseve that the function $\eta$ must be constant, and so, Theorem \ref{geod} applies.

\end{proof}

As a corollary of the previous theorem we obtain the classical Calabi-Bernstein theorem.

\begin{corollary} (Classical Calabi-Bernstein theorem)
The only complete maximal surfaces in the Lorentz-Minkowski spacetime $\mathbb{L}^3$ are the spacelike affine planes.
\end{corollary}

\section*{Acknowledgments}The authors are grateful to the two referees for their deep readings and the suggestions made toward the improvement of this article. The authors are partially supported by the
Spanish MICINN Grant with FEDER funds MTM2013-47828-C2-1-P and by the Junta de Andaluc\'{\i}a Regional Grant P09-FQM-4496.

\end{document}